\documentclass[11pt]{amsart}
\usepackage{amsmath,amssymb}
\usepackage[pdftex]{graphicx}
\usepackage[margin=1.3in]{geometry}

\newcommand{\C} {{\mathbb C}}                              
\newcommand{\R} {{\mathbb R}}                              
\newcommand{\Z} {{\mathbb Z}}                              
\newcommand{\Pj} {{\mathbb P}}                             
\newcommand{\B}{{\mathfrak{b}}}
\newcommand{\Di}{{\mathcal{D}}}
\newcommand{\psor}{\Pj SO_2(\R)}
\newcommand{\psoc}{\Pj SO_2(\C)}
\newcommand{\pslr}{{\Pj SL_2(\R)}}
\newcommand{\pslc}{{\Pj SL_2(\C)}}
\newcommand{\pglr}{{\Pj GL_2(\R)}}
\newcommand{\pglc}{{\Pj GL_2(\C)}}
\newcommand{\slr}{{SL_2(\R)}}
\newcommand{\slc}{{SL_2(\C)}}
\newcommand{\cp}{{\mathbb{CP}^1}}
\newcommand{\rp}{{\mathbb{RP}^1}}
\newcommand{\gi}{G_I}

\newcommand{\Con} {{\rm Config}}
\newcommand{\confa}[2]{\Con^{#2}(\, #1 \,)}
\newcommand{\confb}[3]{\Con^{#2,#3}(\,#1 \,)}

\newcommand{\oM} [1] {\ensuremath{{\mathcal M}_{0,#1}(\R)}}                 
\newcommand{\M} [1] {\ensuremath{{\overline{\mathcal M}}{_{0, #1}(\R)}}}    
\newcommand{\cM} [1] {\ensuremath{{\mathcal M}_{0, #1}}}                    
\newcommand{\CM} [1] {\ensuremath{{\overline{\mathcal M}}_{0, #1}}}         
\newcommand{\Cyc}  [1] {\ensuremath{{\overline{\mathcal Z}}{^{#1}}}}        
\newcommand{\kij}[2]{\ensuremath{{\mathcal K}({#1,#2})}}  
\newcommand{\kbar}[2]{\ensuremath{{\overline{\mathcal K}}({#1,#2})}}        
\newcommand{\kdel}[2]{\ensuremath{{\mathcal K}\langle{#1,#2}\rangle}}       
\newcommand{\mnc}[1]{\ensuremath{\overline{\mathcal{M}}_{#1}(\C)}}
\newcommand{\mc}[1]{\ensuremath{\mathcal{M}_{#1}(\C)}}

\theoremstyle{plain}
\newtheorem{thm}{Theorem}
\newtheorem{prop}[thm]{Proposition}

\newtheorem{lem}[thm]{Lemma}

\theoremstyle{definition}
\newtheorem*{defn}{Definition}
\newtheorem*{exmp}{Example}

\theoremstyle{remark}
\newtheorem*{rem}{Remark}

\newtheorem*{ack}{Acknowledgments}
\numberwithin{equation}{section}

\begin{document}

\title {Moduli spaces of punctured Poincar\'{e} disks}

\author{Satyan L.\ Devadoss}
\address{S.\ Devadoss: Williams College, Williamstown, MA 01267}
\email{satyan.devadoss@williams.edu}

\author{Benjamin Fehrman}
\address{B.\ Fehrman: University of Chicago, Chicago, IL 60637}
\email{bfehrman@math.uchicago.edu}

\author{Timothy Heath}
\address{T.\ Heath: Columbia University, New York, NY 10027}
\email{timheath@math.columbia.edu}

\author{Aditi Vashist}
\address{A.\ Vashist: University of Michigan, Ann Arbor, MI 48109}
\email{avashist@umich.edu}

\begin{abstract}
The Tamari lattice and the associahedron provide methods of measuring associativity on a line. The real moduli space of marked curves captures the space of such associativity.  We consider a natural generalization by considering the moduli space of marked particles on the Poincar\'{e} disk, extending Tamari's notion of associativity based on nesting. A geometric and combinatorial construction of this space is provided, which appears in Kontsevich's deformation quantization, Voronov's swiss-cheese operad, and Kajiura and Stasheff's open-closed string theory.
\end{abstract}

\maketitle

\baselineskip=17pt

%
%
\section{Motivation from Physics}
\subsection{}

Our story begins with the famous \emph{associahedron} polytope.  In his 1951 thesis, Dov Tamari described the associahedron $K_n$ as the realization of his lattice of bracketings on $n$ letters \cite{tam}.  Independently, in his 1961 thesis, Jim Stasheff constructed a convex curvilinear version of it for use in homotopy theory in connection with associativity properties of $H$-spaces \cite{jds1}. The vertices of $K_n$ are enumerated by the Catalan numbers and its construction as a polytope was given independently by Haiman (unpublished) and Lee \cite{lee}.  Figure~\ref{f:k4w3p}(a) shows the example of the associahedron $K_4$.

\begin{figure}[h]
\includegraphics[width=\textwidth]{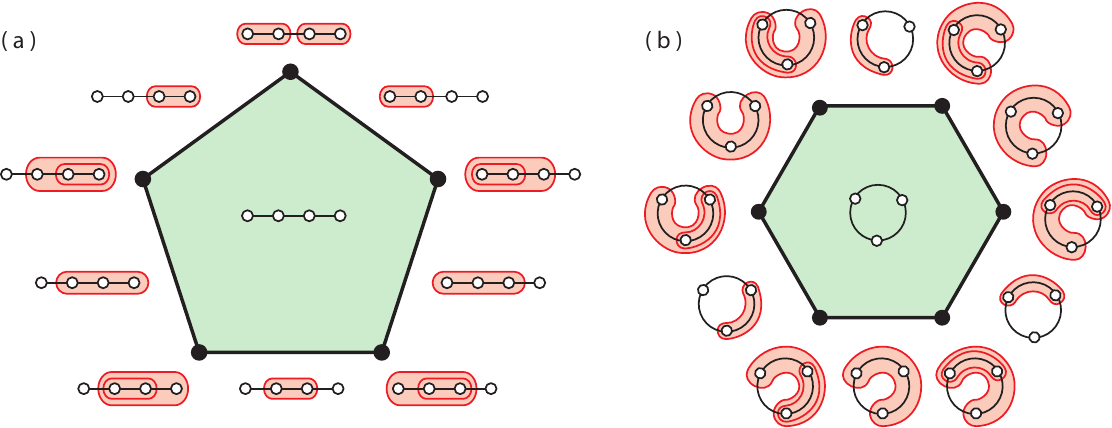}
\caption{(a) Associahedron $K_4$ and (b) cyclohedron $W_3$.}
\label{f:k4w3p}
\end{figure}

\begin{defn}
Let $A(n)$ be the poset of all bracketings of $n$ letters, ordered such that $a \prec a'$ if $a$ is obtained from $a'$ by adding new brackets.   The \emph{associahedron} $K_n$ is a convex polytope of dimension $n-2$ whose face poset is isomorphic to $A(n)$.
\end{defn}

Our interests are based on associahedra as they appear in the world of algebraic geometry. The configuration space of $n$ labeled particles on a manifold $X$ is 
$$\confa X n \ = \ X^n - \Delta,  \ \ \ {\rm where} \ \Delta = \{(x_1,\ldots,x_n) \in X^n \ \ | \ \ \exists \ i,j, \ x_i=x_j \}.$$
The Riemann moduli space ${\mathcal M}_{g,n}$ of genus $g$ surfaces with $n$ marked particles (sometimes called \emph{punctures}) is an important object in mathematical physics, brought to light by Grothendieck in his famous \emph{Esquisse d'un programme}.  A larger framework, based on moduli spaces of \emph{bordered} surfaces of arbitrary genus is considered in \cite{dhv}.
The special case \cM{n} is defined as
$$\cM{n} \ = \ \confa \cp  n \ / \ \pglc \, ,$$
the quotient of the configuration space of $n$ labeled points on the complex projective line by $\pglc$.  There exists a Deligne-Mumford-Knudsen compactification \CM{n} of this space, which plays a crucial role in the theory of Gromov-Witten invariants, symplectic geometry, and quantum cohomology \cite{km}.  

The real points \M{n} of the moduli space are the set of points fixed under complex conjugation; these spaces have importance in their own right, appearing in areas such as $\zeta$-motives \cite{gm}, phylogenetic trees \cite{dm}, and Lagrangian Floer theory \cite{fuk}.   The relationship between \M{n} and the associahedron is given by the subsequent important result:

\begin{thm} \cite[Section 3]{dev1} \label{t:realmod}
The real moduli space of $n$-punctured Riemann spheres
$$\oM{n} \ = \ \confa \rp  n \ / \ \pglr$$
has a Deligne-Mumford-Knudsen compactification \M{n}, resulting in an $(n-3)$-manifold tiled by $(n-1)!/2$ copies of the $K_{n-1}$ associahedron.
\end{thm}

\subsection{}

There are numerous generalizations of the associahedron currently in literature.  The closest kin is the \emph{cyclohedron} polytope, originally considered by Bott and Taubes in relation to knot invariants \cite{bt}.  Figure~\ref{f:k4w3p}(b) shows the example of the 2D cyclohedron $W_3$.

\begin{defn}
Let $B(n)$ be the poset of all bracketings of $n$ letters \emph{arranged in a circle}, ordered such that $b \prec b'$ if $b$ is obtained from $b'$ by adding new brackets.   The \emph{cyclohedron} $W_n$ is a convex polytope of dimension $n-1$ whose face poset is isomorphic to $B(n)$.
\end{defn}

And just as the associahedron tiles \M{n}, there is an analogous manifold tiled by the cyclohedron. Indeed, Armstrong et al. \cite{small03} consider a collection of such moduli spaces (based on blowups of Coxeter complexes), each tiled by different analogs of the associahedron polytope, called \emph{graph associahedra} \cite{cd}.  

\begin{thm} \cite[Section 2]{dev2} \label{t:cycmod}
The moduli space $\Cyc{n}$ is the (Fulton-MacPherson) compactification of \, $\confa {S^1} n \ / \ S^1$, tiled by $(n-1)!$ copies of the cyclohedron $W_n$.
\end{thm}

Both \M{n} and \Cyc{n} consider how particles move and collide on the circle (viewed as either $\rp$ or $S^1$, depending on the group action).  Similarly \CM{n} encapsulates particle collisions on the sphere $\cp$.  In this article, we consider the Poincar\'{e} disk, a concrete playground where particles in the interior can collide similar to \CM{n} and on the boundary similar to \M{n} and \Cyc{n}.  Indeed, several others have considered a version of this space of punctured disks:  Kontsevich in his work on deformation quantization \cite{kon}, Kajiura and Stasheff from a homotopy algebra viewpoint \cite{ks}, Voronov from an operadic one \cite{vor}, and Hoefel from that of spectral sequences \cite{hoe}.   We claim that this space naturally extends the notion of Tamari's associativity, from particles on lines to particles on disks.

This is a survey article on the moduli space of punctured Poincar\'{e} disks, from a geometric and combinatorial viewpoint.  Section~\ref{s:disk} introduces the foundational setup, whereas Section~\ref{s:fm} provides the famous Fulton-MacPherson compactification based on iterated blowups, along with discussing several examples.  A local construction of this space using group actions on bubble-trees is given in Section~\ref{s:groups}, and Section~\ref{s:combin} ends with some combinatorial results.

%
%
\section{Particles on the Poincar\'{e} Disk} \label{s:disk}
\subsection{}

There exists a natural action of $\pslr$ on the upper halfplane $\mathbb{H}$ (along with the point at infinity) given by
$$\left(\begin{array}{cc} a & b \\ c & d \end{array}\right) \cdot x \ = \ \frac{ax+b}{cx+d}.$$
The diffeomorphism $z \rightarrow (z-i)(z+i)^{-1}$
extends this to an action of $\pslr$ on the Poincar\'{e} disk $\Di$, where infinity is on the boundary of $\Di$. The classical $KAN$ decomposition of $\slr$ is given by
$$K = SO_2(\mathbb{R})
\hspace{1cm}
A = \left\{\left(\begin{array}{cc} a & 0 \\ 0 & a^{-1} \end{array} \right) \;\biggl\lvert\; a > 0\right\} 
\hspace{1cm}
N = \left\{\left(\begin{array}{cc} 1 & x \\ 0 & 1 \end{array}\right) \;\biggl\lvert\; x\in\mathbb{R}\right\}.$$
Analogously, the corresponding decomposition of $\pslr$ is
\begin{equation}
\psor \cdot A \cdot N.
\label{e:kan}
\end{equation}
 The following is a classical result; see \cite[Chapter 5]{rat} for details.

\begin{prop} \label{p:psl}
The action of $\pslr$ preserves the boundary and the interior of $\Di$.  Moreover, 
\begin{enumerate}
\item[1.] if $\mu$ is a point in the interior of $\Di$, then for each element $x$ in the interior of $\Di$, there exists a unique element $\sigma$ in $A\cdot N$ satisfying $\sigma\cdot x = \mu$, and
\item[2.] for each point $y$ on the boundary of $\Di$, there exists a unique element $\sigma$ in $\psor$ satisfying $\sigma\cdot y=\infty$.
\end{enumerate}
\end{prop}

\noindent The action of $\pglr$ is naturally characterized by $\pslr$:  The element
$$r=\left(\begin{array}{cc} 1 & 0 \\ 0 & -1 \end{array}\right)$$
 in $\pglr\setminus\pslr$ acts on $\Di$ as the reflection about the geodesic connecting zero to $\infty$.  The coset partition $\left\{\pslr,\;r\cdot\pslr\right\}$ decomposes
$\pglr$ as
\begin{equation} \label{e:pglr}
\mathbb{Z}_2\cdot \psor \cdot A \cdot N
\end{equation}
where $\mathbb{Z}_2=\left\{I,\;r\right\}$.   Figure~\ref{f:pslr} shows the action of $\pslr$ on $\Di$ based on the decomposition given by Eq.~\eqref{e:kan}, where the Poincar\'{e} disk $\Di$ in shaded to help display the action.  Part (a) shows $\Di$ with four of its points labeled, along with the action of (b) reflection $r$, (c) rotation $\psor$, and (d) interior movement $A \cdot N$.

\begin{figure}[h]
\includegraphics[width=\textwidth]{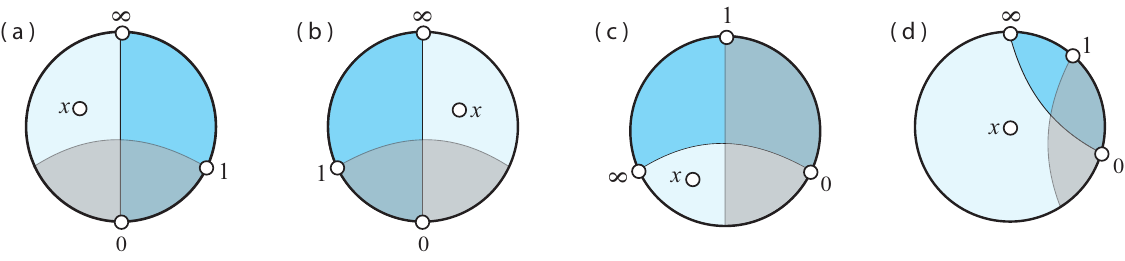}
\caption{Actions on the Poincar\'{e} disk.}
\label{f:pslr}
\end{figure}

\subsection{}

The main interest of this paper is the configuration space
$$\confb \Di n m \ = \ (\Di^n \times \partial \Di^m) - \Delta$$
such that $\Delta$ is the collection of points $(p_1,\ldots,p_n,q_1,\ldots,q_m) \in \Di^n\times\partial\Di^m$ \ where
\begin{enumerate}
\item[1.] (\emph{interior collision}) there exist $i,j$ such that $p_i=p_j$ , or
\item[2.] (\emph{boundary collision})  there exist $i,j$ such that $q_i=q_j$, or
\item[3.]  (\emph{mixed collision})  there exists $i$ such that $p_i \in \partial\Di$.
\end{enumerate}
Points of $\confb \Di n m$ have $n$ distinct labeled particles in the interior of $\Di$ and $m$  distinct labeled particles confined to the boundary of $\Di$.  Since each element of $\pslr$ preserves the boundary and interior of $\Di$, the action of $\pslr$ on $\Di$ naturally extends to an action on $\confb \Di n m$.  Thus, the following is well-defined:

\begin{defn} \label{d:moduli}
The moduli space of punctures on the Poincar\'{e} disk $\Di$ is the quotient
$$\kij{n}{m} \ = \ \confb \Di n m \, / \, \pslr.$$
\end{defn}

We will be concerned mostly with the case when $n, m \geq 1$, with at least one particle $\mu$ in the interior of $\Di$ and one particle $\infty$ on the boundary of $\Di$.  Consider 
$$\confb {\Di}{n-1}{m-1} - \Delta^*$$
such that $\Delta^*$ is the collection of points $(p_1,\ldots,p_{n-1},q_1,\ldots,q_{m-1}) \in \Di^n\times\partial\Di^m$ \ where
\begin{enumerate}
\item[1.] (\emph{interior collision}) there exist $i,j$ such that $p_i=p_j$ or $p_i = \mu$, or
\item[2.] (\emph{boundary collision}) there exist $i,j$ such that $q_i=q_j$ or $q_i = \infty$, or
\item[3.] (\emph{mixed collision}) there exists $i$ such that $p_i \in \partial\Di$.
\end{enumerate}

\begin{prop} \label{p:fixedpts} 
The moduli space $\kij n m$ admits a natural description as the space $\confb \Di {n-1} {m-1} - \Delta^*$.
\end{prop}

\begin{proof}
It follows directly from Proposition~\ref{p:psl} that each orbit under the $\pslr$ action on $\Di$ can be uniquely represented by a particle configuration of this form:  Fixing an interior particle kills the $A\cdot N$ action, and the remaining $\psor$ rotation of Eq.~\eqref{e:kan} is addressed by fixing a boundary particle. 
\end{proof}

\begin{rem}
Although we do not discuss it here, the moduli space $\kij n 0$ of $n$ interior particles on the disk can be interpreted in terms of fixed particles, analogously to Proposition~\ref{p:fixedpts}.  The $\pslr$ action fixes one of the $n$ particles at $\mu$ using $A\cdot N$  and confines another particle to the geodesic connecting $\mu$ to $\infty$ using $\psor$.  The resulting moduli space can be shown to be homeomorphic to $\kij n 0$.  The $\kij 0 m$ case is considered in Section~\ref{s:combin}.
\end{rem}

\subsection{}

So far, the particles in $\kij n m$ are not allowed to collide.  The most natural manner of embracing collisions is by including the places of collision $\Delta$ that were removed by the configuration space.  We define the \emph{naive compactification} of $\kij n m$ as the inclusion of the diagonal $\Delta$ with $\kij n m$, denoted as $\kdel n m$.  A visual notation is now introduced to label the particle collisions on $\Di$ based on \emph{arcs}.

\begin{defn}
An \emph{arc} is a curve on $\Di$ such that its endpoints are on the boundary of $\Di$, it does not intersect itself nor the particles, and encloses at least one interior particle or two boundary particles.
A \emph{loop} is an arc with its endpoints identified, enclosing multiple interior particles.\footnote{Abusing terminology, we will refer to both arcs and loops as simply ``arcs''.}
Two arcs are \emph{compatible} if there are curves in their respective isotopy classes which do not intersect.  
\end{defn}

From the definition of $\Delta$, three types of collisions emerge.  Figure~\ref{f:notate} illustrates collisions in $\kdel 7 6$, part (a) showing a diagram of the disk with the appropriate particles, with examples of an (b) interior collision, (c) boundary collision, and (d) mixed collision.
\begin{figure}[h]
\includegraphics[width=\textwidth]{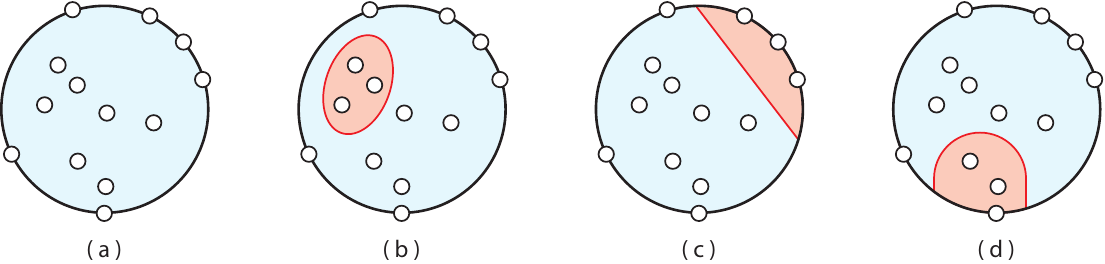}
\caption{Diagrams of arcs on $\Di$ corresponding to collisions of particles.}
\label{f:notate}
\end{figure}

\subsection{}

Let us consider some low-dimensional examples, where \emph{black} particles will represent the fixed particles $\mu$ and $\infty$ on the Poincar\'{e} disk. 

\begin{exmp}
Since the action of $\pslr$ fixes a particle on the interior and one on the boundary, the moduli space $\kdel 1 1$ is a point.
\end{exmp}

\begin{exmp}
The left side of Figure~\ref{f:k12-21} displays $\kdel 1 2$, which is a circle with a vertex (a), corresponding to the free particle on the boundary colliding with the fixed particle $\infty$.
\end{exmp}

\begin{figure}[h]
\includegraphics[width=\textwidth]{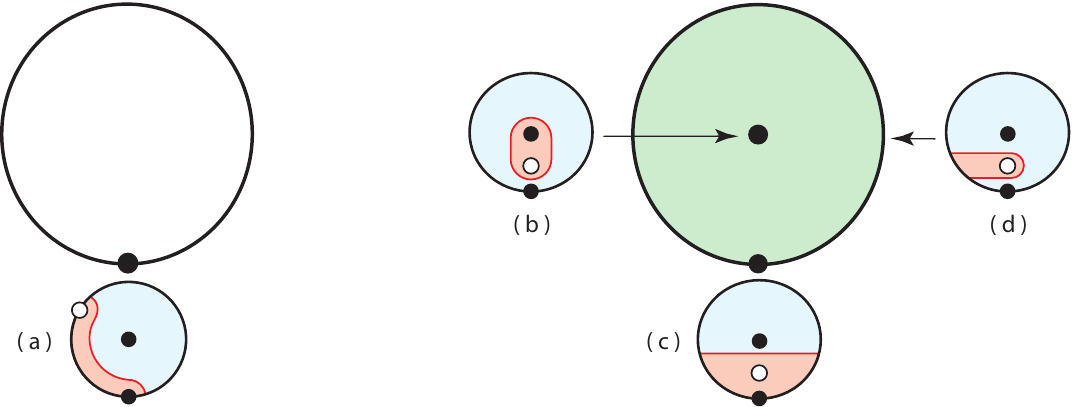}
\caption{$\kdel 1 2$ and $\kdel 2 1$}
\label{f:k12-21}
\end{figure}

\begin{exmp}
The right side of Figure~\ref{f:k12-21} displays $\kdel 2 1$ as a disk with two vertices.  Here, there are three types of collisions signified by three cells: vertex (b) is the free particle colliding with the interior fixed particle $\mu$, vertex (c) is the free particle colliding with $\infty$, and circle (d) is the free particle colliding with the boundary of $\Di$.  Note that this last condition is considered a \emph{collision} since $\Delta^*$ includes $\partial \Di$.
\end{exmp}

\begin{exmp}
Figure~\ref{f:k13} shows the two-torus $\kdel 1 3$.  It has three circles on it, all of which are copies of $\kdel 1 2$,  given by collisions (a, b, c), along with a vertex (d) where all the circles meet. 
\end{exmp}

\begin{figure}[h]
\centerline{\includegraphics[width=.5\textwidth]{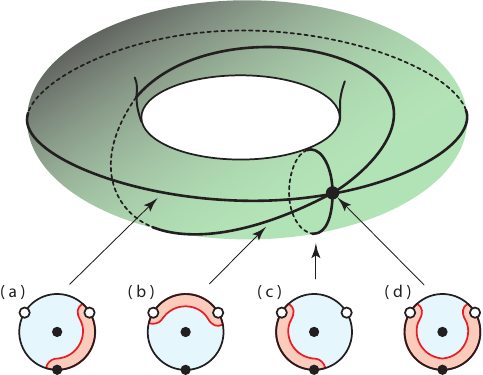}}
\caption{The two-torus $\kdel 1 3$.}
\label{f:k13}
\end{figure}

\begin{exmp}
Figure~\ref{f:k22} displays $\kdel 2 2$ as a \emph{solid} two-torus with diagrammatic labelings of the collisions.   It has one vertex (e) where all the free particles collide with $\infty$, three $\kdel 1 2$ circles (b, c, f), a $\kdel 2 1$ disk (d), and the space $\kdel 1 3$ as the two-torus boundary  (a) of $\kdel 2 2$.
\end{exmp}
 
\begin{figure}[h]
\centerline{\includegraphics[width=.7\textwidth]{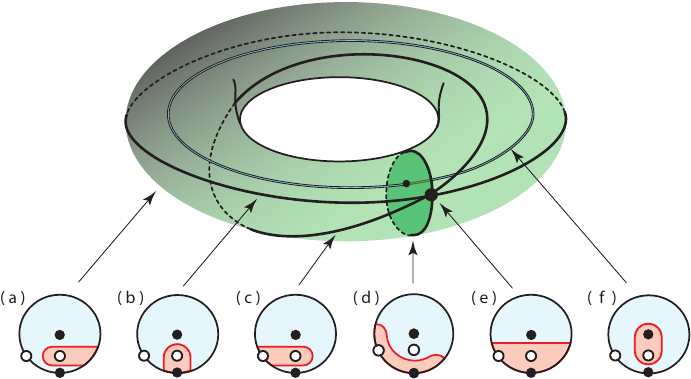}}
\caption{The solid two-torus $\kdel 2 2$.}
\label{f:k22}
\end{figure}

%
%
\section{The Fulton-McPherson Compactification} \label{s:fm}
\subsection{}

Although the naive compactification $\kdel n m$ encapsulates particle collisions, there are other compactifications which contain more useful information.  In general, compactifying $\kij n m$ enables the particles to collide and a \emph{system} is introduced to record the \emph{directions} particles arrive at the collision.  In the work of Fulton and MacPherson \cite{fm}, this method is brought to rigor in the algebro-geometric context.\footnote{The real analog of the Fulton-MacPherson is given by the Alexrod-Singer compactification \cite{as}.}   We now care not just about particle collisions but the space of \emph{simultaneous particle collisions}.

The language to construct this compactification is the algebro-geometric notion of a \emph{blowup}.  Since we are manipulating \emph{real} manifolds with boundary, there are two notions of blowups which are needed:  Let $Y$ be a manifold with boundary.  For a subspace $X$ in the \emph{interior} of $Y$ , the \emph{blowup} of $Y$ along $X$ is obtained by first removing $X$ and replacing it with the sphere bundle associated to the normal bundle of $X \subset Y$. We then projectify the bundle.  

\begin{exmp}
Figure~\ref{f:blowup-ball} displays an example of the blowup of the plane $Y$ along a point $X$:  Part (a) shows the set of normal directions and part (b) shows the point replaced by the sphere bundle.  Part (c) shows the result of the antipodal map along this bundle (in red), where $X$ has now been replaced by $\rp$.  Indeed, the blow up keeps track of the \emph{projective direction} in which one approaches $X$.
\end{exmp}

\begin{figure}[h]
\includegraphics[width=\textwidth]{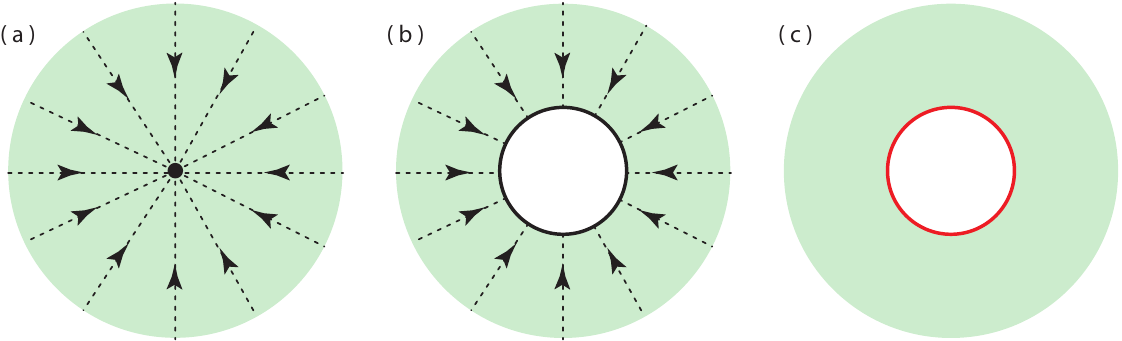}
\caption{(a) Set of directions to a point $X$, (b) blowing up along $X$, and (c) antipodal gluing.}
\label{f:blowup-ball}
\end{figure}

For a subspace $X$ along the \emph{boundary} of $Y$, the \emph{boundary blowup} of $Y$ along $X$ is obtained by first removing $X$ and replacing it with the sphere bundle associated to the normal bundle of $X \subset Y$.   We then projectify the bundle only along its intersection with the boundary.

\begin{exmp}
Figure~\ref{f:blowup-disk} provides an example of the blowup of the halfplane $Y$ along a point $X$ on its boundary:  Part (a) shows the set of normal directions and part (b) shows the point replaced by the normal bundle.  Part (c) shows the result of the antipodal map along the intersection of this bundle with the boundary, in this case identifying the two highlighted points.  In general, a boundary blowup keeps track of the directions in which one approaches $X$ from the interior of $Y$, and the \emph{projective direction} in which one approaches $X$ along the boundary of $Y$.
\end{exmp}

\begin{figure}[h]
\includegraphics[width=\textwidth]{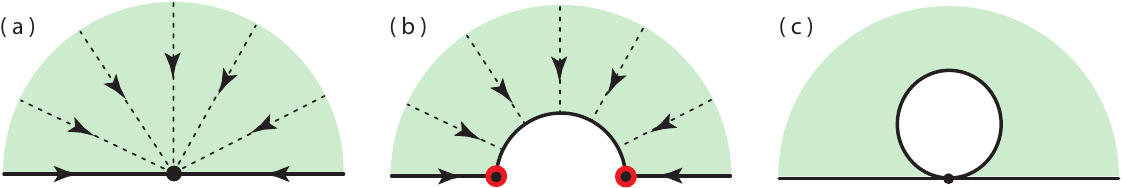}
\caption{(a) Set of directions, (b) blowing up along the point, and (c) gluing along the bounding hyperplane.}
\label{f:blowup-disk}
\end{figure}

\subsection{}

A general collection of blowups is usually non-commutative in nature; in other words, the order in which spaces are blown up is important.  For a given hyperplane arrangement, De Concini and Procesi \cite{dp} establish the existence and uniqueness of a \emph{minimal building set}, a collection of subspaces for which blowups commute for a given dimension.  In the case of the arrangement $X^n - \confa {X}{n}$, their procedure yields the Fulton-MacPherson compactification of $\confa X n$. 

\begin{defn}
The \emph{minimal building set} $\B(n,m)$ of $\kdel n m$ is the collection of elements in $\kdel n m$ labeled with a single arc on $\Di$.
\end{defn}

The elements of $\B(n,m)$ are partitioned according to the magnitude of the collisions they represent.  Each element represents a configuration where $i$ interior particles and $b$ boundary particles have collided, and the sum $2i+b$ determines the dimension of this element.  One important note is that real blowups along cells which are already codimension one do not alter the topology of the manifold.  Thus, we discount codimension one elements from the building set $\B(n,m)$.  The work of De Concini and Procesi give us the ensuing result:

\begin{thm}
The Fulton-McPherson compactification $\kbar n m$ is obtained from $\kdel n m$ by the iterated sequence of blowups of
$\B(n,m)$ in increasing order of dimension. 
\end{thm}

In the language of algebraic geometry, the effect of a blowup replaces the cell with an \emph{exceptional divisor} of the resulting manifold.  From a topological perspective, a blowup of $Y$ along $X$ ``promotes'' $X$ to become a codimension one cell of $Y$.  An analog from the world of polytopes is the notion of truncation:  Truncating any face of a polytope introduces a new facet, and the truncation of a current facet does not change the combinatorics of the polytope.

\subsection{}

Let us consider some low-dimensional examples of the construction of $\kbar n m$ using iterated blowups.

\begin{exmp}
The left side of Figure~\ref{f:k12-21} displays $\kdel 1 2$, which is a circle with a vertex (a). Since cell (a) is codimension one, no blowups are necessary, and $\kbar 1 2$ is the same as $\kdel 1 2$.
\end{exmp}

\begin{exmp}
The building set $\B(2,1)$ consists of two points, given by the two vertices of Figure~\ref{f:eye-glue}(a); compare with the labels (b) and (c) of Figure~\ref{f:k12-21}.  The blowups along these two cells are first accomplished by replacing these points with the sphere bundle along the normal bundle.  The interior vertex is replaced with a circle, and the boundary vertex with an arc, as in Figure~\ref{f:eye-glue}(b).  Performing the projective identification is shown in part (c), where the interior circle has an antipodal map (in red), and the two points on the boundary are now identified.
\end{exmp}

\begin{figure}[h]
\includegraphics[width=\textwidth]{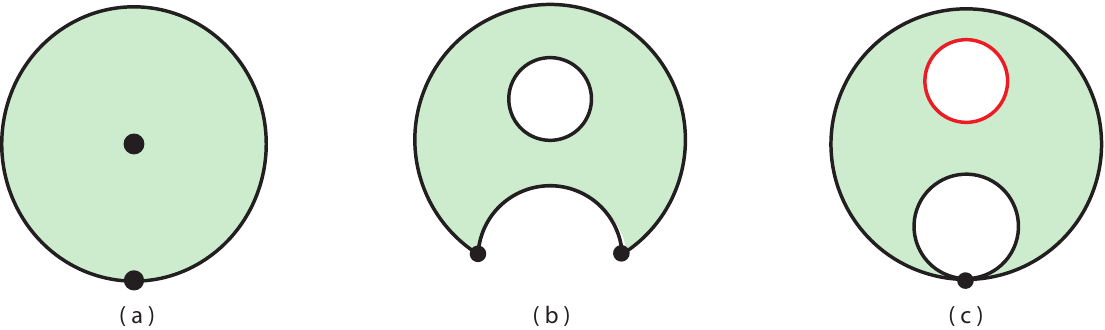}
\caption{(a) $\kdel 2 1$, (b) the eye of Kontsevich, and (c) $\kbar 2 1$.}
\label{f:eye-glue}
\end{figure}

\begin{rem}
The picture in Figure~\ref{f:eye-glue}(b) appears in the work by Kontsevich on deformation quantization \cite[Figure 7]{kon}, where he calls this the ``\emph{eye}''.
\end{rem}

\begin{exmp}
Figure~\ref{f:k13} shows $\kdel 1 3$ as a two-torus, with the building set $\B(1,3)$ containing only one vertex, labeled by (d).  Blowing up this vertex results in the connected sum of the two-torus with an $\R\Pj^2$, given in Figure~\ref{f:k13-tiles}(a).  Part (b) of this figure displays the two tiling hexagons along with their gluing map used to construct $\kbar 1 3$.
\end{exmp}

\begin{figure}[h]
\includegraphics[width=\textwidth]{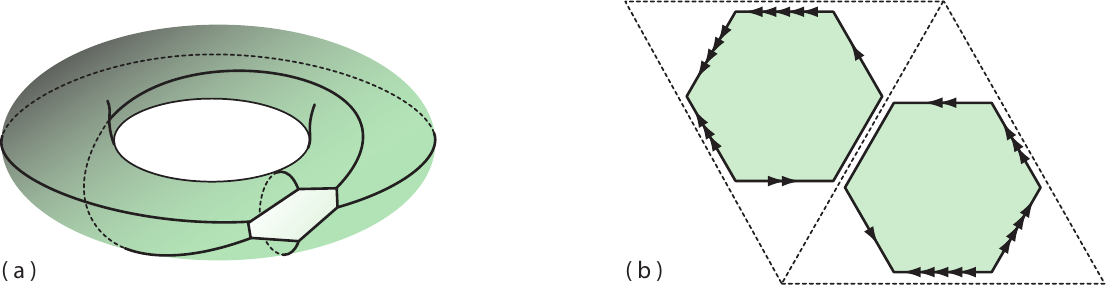}
\caption{(a) $\kbar 1 3$ obtained from (b) gluing two hexagons.}
\label{f:k13-tiles}
\end{figure}

\begin{exmp}
The building set $\B(2,2)$ consists of four elements, shown in Figure~\ref{f:k22}: three circles (b, c, f), and one vertex (e).  
A redrawing of this solid two-torus cut open along the disk (e) is given in Figure~\ref{f:k22-iterate}(a).  The construction of $\kbar 2 2$ follows from iterated blowups in increasing order of dimension.  Figure~\ref{f:k22-iterate}(b) shows the blowup along the vertex on the boundary, and part (c) shows the blowup along the three circles.  Although not displayed in these figures, there are gluing maps from the blowups identifying faces of these cells.  The combinatorics of this space is expanded upon in Figure~\ref{f:k22-bar}.
\end{exmp}

\begin{figure}[h]
\includegraphics[width=\textwidth]{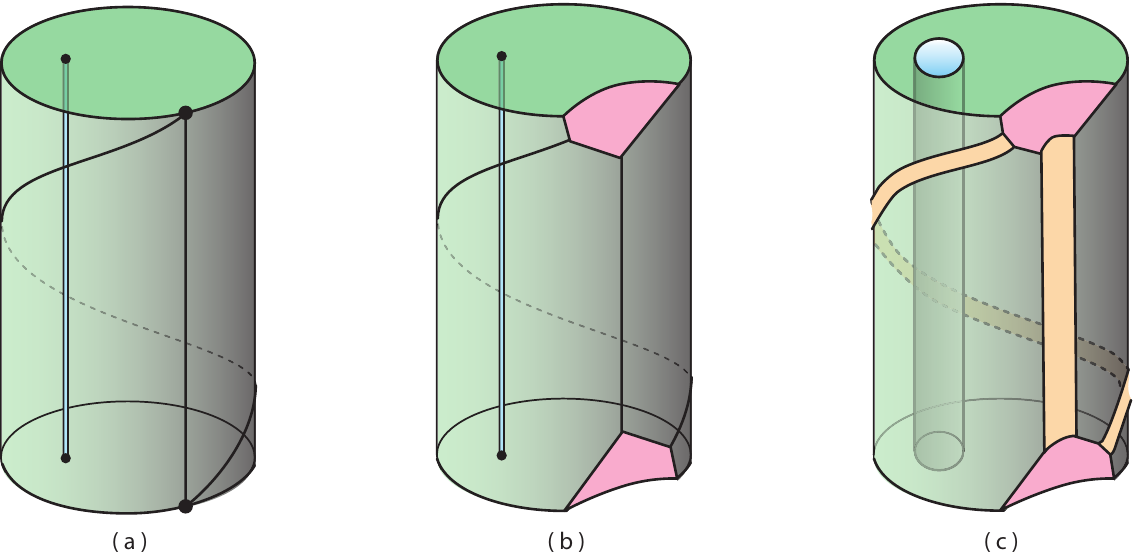}
\caption{Iterated truncation of $\kij 2 2$ resulting in $\kbar 2 2$.}
\label{f:k22-iterate}
\end{figure}

\subsection{}

Recall that $\kdel n m$ is stratified by compatible arcs on the disk $\Di$ representing collisions.  For the compactified moduli space $\kbar n m$, the stratification is given by compatible \emph{nested} arcs, the notion coming from the work of Fulton-MacPherson \cite{fm}.  Consider Figure~\ref{f:nested}:  parts (a) and (b) show valid compatible arcs, representing cells of $\kbar 7 5$.  Part (c) shows compatible nested arcs of the disk, representing a cell of $\kbar 7 5$.   It is this nesting which gives rise to a generalized notion of Tamari's associativity.  

\begin{figure}[h]
\includegraphics[width=\textwidth]{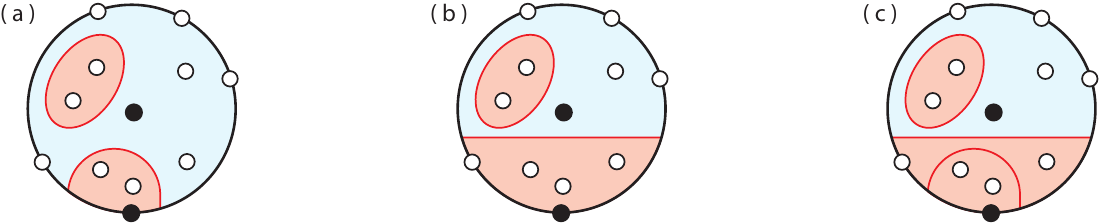}
\caption{Compatible arcs (a) and (b) of $\kdel 7 5$, along with nested arcs (c) of $\kbar 7 5$.}
\label{f:nested}
\end{figure}

We summarize as follows:

\begin{thm}
The space $\kbar n m$ of marked particles on the Poincar\'{e} disk is a compact manifold of real dimension $2n + m -3$, naturally stratified by nested compatible arcs, where $k$ arcs on $\Di$ correspond to a codimension $k$ face of $\kbar n m$. 
\end{thm}

This space can be viewed from the perspective of open-closed string field theory, where Figure~\ref{f:notate} is reinterpreted as displaying (b) open strings, (c) closed strings, and (d) open-closed strings.   There is a rich underlying operadic perspective to this space, brought to light by the works of others such as Voronov \cite{vor}, Kajiura-Stasheff \cite{ks}, and Hoefel \cite{hoe}.  The \emph{swiss-cheese} operad structure of Voronov extends the little disks operad by incorporating boundary pieces, envisioned by replacing the particles of Figure~\ref{f:notate} by loops (interior) and arcs (boundary).  

Kajiura and Stasheff introduced the open-closed homotopy algebras (OCHA) by adding other operations to an $A_\infty$ algebra over an $L_\infty$ algebra.    The meaning of these extra operations amounts to making a closed string become open; from our viewpoint, it signifies the collision of the interior particle with the boundary of $\Di$.  Finally, Hoefel has shown that the OCHA operad is quasi-isomorphic to the swiss-cheese operad based on calculations of spectral sequences.

%
%
\section{Group Actions on Screens} \label{s:groups}
\subsection{}

Instead of a global perspective, based on blowups of cells in the building set, there is a local, combinatorial perspective in which to construct this moduli space.  Fulton and MacPherson describe their compactification from the collision perspective as follows:
As points collide along a $k$-dimensional manifold, they land on a $k$-dimensional \emph{screen}, which is identified with the point of collision.  These screens have been dubbed \emph{bubbles}, and the compactification process as \emph{bubbling}; see \cite{pw} for details from an analytic viewpoint.  Now these particles on the screen are themselves allowed to move and collide, landing on higher level screens.   Kontsevich describes this process in terms of a magnifying glass:  On any given level, only a configuration of points is noticeable; but one can zoom-in on a particular point and peer into its screen, seeing the space of collided points.

As the stratification of this space is given by collections of compatible nested arcs, the compactification of this space can be obtained by the \emph{contraction} of these arcs --- the degeneration of a decomposing curve $\gamma$ as its length collapses to zero.  There are three possible results obtained from a contraction, as in Figure~\ref{f:bubs}.  Part (a) shows the contraction of a loop (closed arc) capturing interior collisions, resulting in what we call a \emph{sphere bubble}.  Notice the arc has now been identified with a marked particle. Part (b) displays contractions of an arc capturing boundary collisions; the resulting bubble is called a \emph{flat bubble}.  Notice that all the particles are only on the boundary of this new disk.  Finally, (c) shows contractions for a mixed collision, called a \emph{punctured bubble}.  

\begin{figure}[h]
\includegraphics[width=\textwidth]{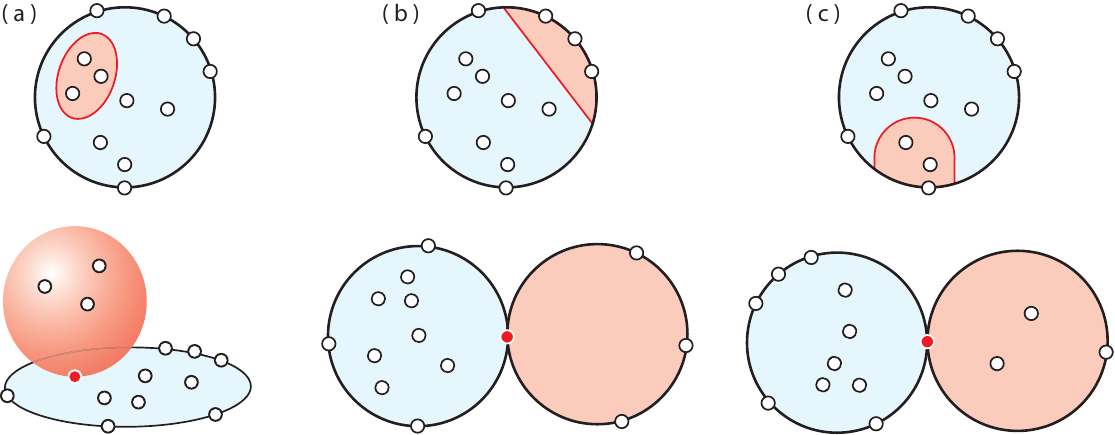}
\caption{(a) Sphere bubble, (b) flat bubble, and (c) punctured bubble.}
\label{f:bubs}
\end{figure}

Given a collection of compatible nested arcs, collapsing all the arcs results in several bubbles, sometimes refereed to as a \emph{bubble tree}.   There exists a natural dual perspective in which to visualize the cells of $\kbar n m$, as given by Figure~\ref{f:trees}.   Based on the notation of Hoefel \cite{hoe}, we obtain \emph{partially-planar} trees, where the interior particles correspond to leaves in $\R^3$ and boundary particles to leaves restricted to the $xy$-plane. The spatial edges (red zig-zags) are allowed to move freely, whereas the cyclic ordering of the planar ones (black lines) are determined by the disk.  The natural operad composition maps for such trees are given in \cite[Section 4]{hoe}.  If we choose a boundary particle to be fixed at $\infty$, then \emph{rooted} trees arise.

\begin{figure}[h]
\includegraphics[width=\textwidth]{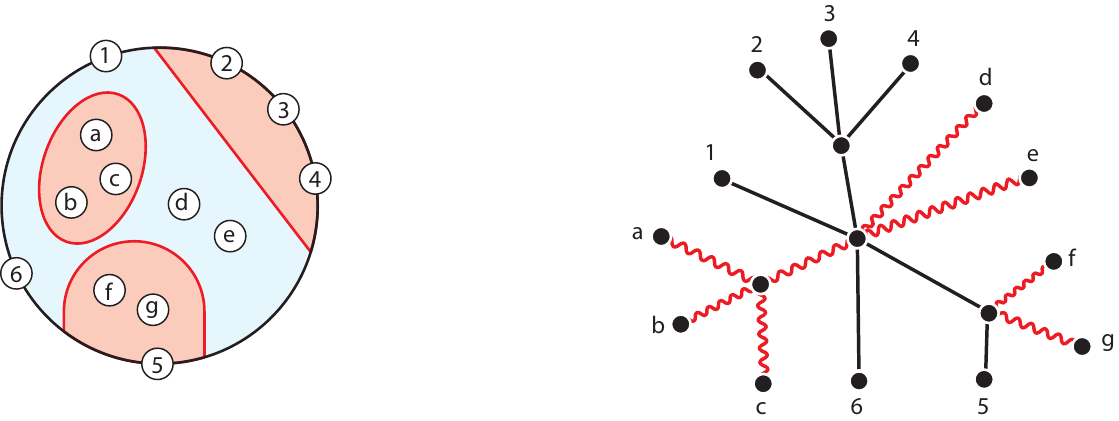}
\caption{Duality between arcs on the disk and partially-planar trees.}
\label{f:trees}
\end{figure}

\subsection{}

These three types of bubbles are the \emph{screens} described by Fulton and MacPherson.  According to the compactification, there is an action of an automorphism group on each screen (bubble).   In order to describe the groups, a slight detour is taken.  Similar to $\pslr$, the matrix group $\pslc$ acts on $\cp$, considered as $\C \cup \left\{\infty\right\}$.  The decomposition of $\slc$ analogous to Eq.~\eqref{e:kan} is given by
$$K_\C=SO_2(\C)
\hspace{1cm}
A_\C=\left\{\left(\begin{array}{cc} z & 0 \\ 0 & z^{-1} \end{array}\right) \;\biggl\lvert\; z\in\C^{\times}\right\}
\hspace{1cm}
N_\C=\left\{\left(\begin{array}{cc} 1 & z \\ 0 & 1 \end{array}\right) \;\biggl\lvert\; z\in\C\right\}.$$
The corresponding decomposition of $\pslc$ is given by
$$\mathbb{P}SO_2(\mathbb{C})\cdot A_\mathbb{C}\cdot N_\mathbb{C}.$$
An interest of this paper is the particular group
\begin{equation}
\label{e:gi}
\gi \ = \ \Z_2\cdot\psoc\cdot A\cdot N_\C
\end{equation}
acting on $\cp$. Here, $\mathbb{Z}_2=\left\{I,r\right\}$ and $A_\C = S^1 \cdot A$, where $S^1$ is the group of rotations
$$\left\{\left(\begin{array}{cc} e^{i\theta} & 0 \\ 0 & e^{-i\theta}\end{array}\right) \;\biggl\lvert\; \theta\in\mathbb{R}\right\}.$$
\begin{prop} \label{p:psoc}
The action of $\gi$ on $\cp$ is given as follows:
\begin{enumerate}
\item[1.] The element $r$ acts on each circle of constant norm in $\cp$ as the antipodal map.
\item[2.] For each point $x$ in $\cp$, there exists a unique element $\sigma$ in $\psoc$ satisfying $\sigma\cdot x=\infty$.
\item[3.] For any two elements $x,y$ in $\cp \setminus \{ \infty \}$, there exists a unique element $\sigma$ in $A\cdot N_\mathbb{C}$ satisfying $\sigma\cdot x=0$ and $||\sigma\cdot y||=1$.
\end{enumerate}
\end{prop}

\begin{proof}
Since the elements of $\psoc$ act homeomorphically on $\cp$, the action of each element sends a unique point
$x$ in $\cp$ to $\infty$.  Because $SO_2(\mathbb{C})$ is disconnected, this criterion characterizes
each element of $SO_2(\C)$ up to sign.
For two distinct points $x,y$ on $\cp$, each element of $A\cdot N_\C$ corresponds to a dilation and translation of $\cp \setminus \{\infty\}$.  There exists a unique positive scalar $\alpha$ satisfying $\alpha\cdot||x-y||=1$ and a unique $\beta$ in $\C$ satisfying $\beta+\alpha\cdot x=0$.
This specifies $\sigma$ uniquely because $A\cdot N_\C$ is connected.
\end{proof}

The action of $\gi$ on $\cp$ extends to an action on $\confa \C n$.  By Proposition~\ref{p:psoc}, the action fixes two of the $n$ particles, and restricts the third particle to a (projective) circle's worth of freedom.  The following is a natural object to consider:

\begin{defn}
The moduli space $\mc n$ is the quotient space
$$\mc n \ = \ \confa \cp n \ / \ \gi \, .$$
Denote by $\mnc n$ the natural Fulton-MacPherson compactification of this space.
\end{defn}

\begin{rem}
There is a difference between $\mnc n$ and the classical moduli space of curves \CM{n}.  The former is based on real blowups while the latter is based on complex blowups.  Indeed, the group $\pglc$ has six real dimensions, whereas $\gi$ only has five.  Thus, for instance, $\mnc 3$ is equivalent to $\rp$ whereas $\CM{3}$ is simply a point.
\end{rem}

\begin{lem} \label{l:groups}
Each type of bubble receives a different group action:
\begin{enumerate}
\item[1.] The group $\gi$ acts on sphere bubbles.
\item[2.] The group $\pglr$ acts on flat bubbles.
\item[3.] The group $\pslr$ acts on punctured bubbles.
\end{enumerate}
\end{lem}

\begin{proof}
The action of $\gi$ on the sphere bubbles is given by Proposition~\ref{p:psoc} above, where it mimics the real blowup of particle collisions on the plane, as displayed in Figure~\ref{f:blowup-ball}.  The action of $\pglr$ on flat bubbles is based on the decomposition given in Eq.~\eqref{e:pglr}.  There is a $\Z_2$ component which identifies each flat bubble with its mirror image about the geodesic connecting zero to $\infty$ on $\Di$, replicating the boundary blowup structure shown in Figure~\ref{f:blowup-disk}.  Finally, $\pslr$ is the natural group action on the punctured bubbles (Poincar\'{e} disks) from Proposition~\ref{p:psl}.
\end{proof}

Since the arcs on $\Di$ are compatible and thus nonintersecting, each bubble-tree corresponds to a product of smaller moduli spaces.   The group actions from Lemma~\ref{l:groups} yield the subsequent result:

\begin{thm} \label{t:product}
Each bubble gives rise to a compactified moduli space:
\begin{enumerate}
\item[1.] Sphere bubbles with $n$ particles produce $\mnc n$.
\item[2.] Flat bubbles with $m$ particles produce $\M{m}$.
\item[3.] Punctured bubbles with $n$ interior and $m$ boundary particles produce $\kbar n m$.
\end{enumerate}
Moreover, each diagram of $\Di$ with compatible nested arcs corresponds to a product of moduli spaces, with each bubble contributing a factor.
\end{thm}

\begin{exmp}
Consider $\kbar 2 2$ given in Figure~\ref{f:k22-iterate}; it is redrawn (with a twist of $\pi$ of the top disk) in Figure~\ref{f:k22-bar} on the left.  The colors and the labels of these figures are coordinated to match. 
\begin{enumerate}
\item[1.] There is one interior collision (f) resulting in a sphere bubble adjoined to a punctured bubble, associated to the space $\kbar 1 2 \times \mnc{3}$.  Since both $\kbar 1 2$ and $\mnc{3}$ are topological circles, the resulting divisor is a two-torus.  
\item[2.] There is one boundary collision (d) producing a punctured bubble adjoined to a flat bubble, giving rise to $\kbar 2 1 \times \M{3}$.  Since $\M{3}$ is a point, the resulting divisor is simply $\kbar 2 1$.
\item[3.] Two mixed collisions (b, c) produce $\kbar 1 2 \times \kbar 1 2$, yielding a two-torus, the product of two circles.
\item[4.] Two mixed collisions (a, e) produce $\kbar 1 1 \times \kbar 1 3$.  Since $\kbar 1 1$ is a point, we are left with $\kbar 1 3$ of Figure~\ref{f:k13-tiles}(a).
\end{enumerate}
\end{exmp}

\begin{figure}[h]
\includegraphics[width=\textwidth]{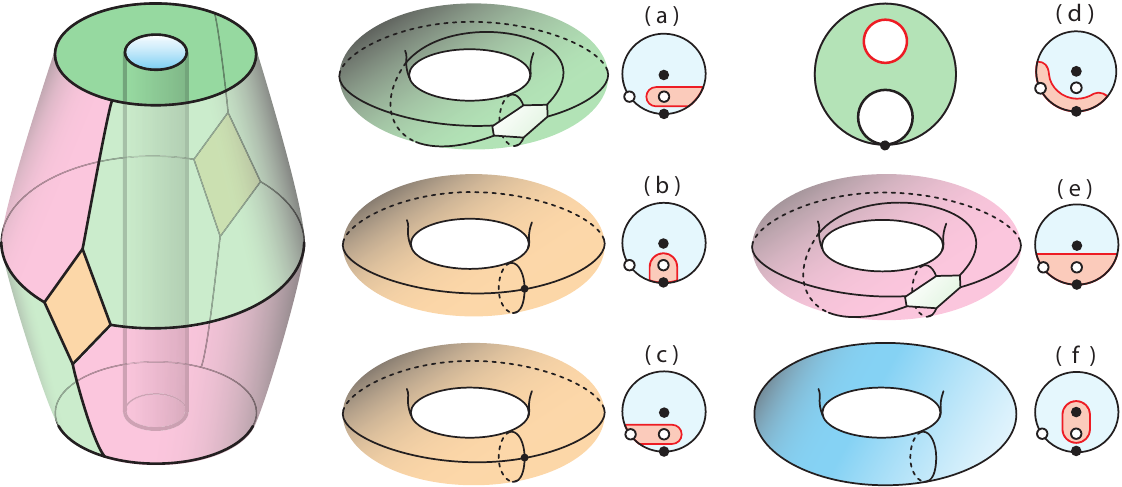}
\caption{$\kbar 2 2$ and its exceptional divisors.}
\label{f:k22-bar}
\end{figure}

%
%
\section{Combinatorial Results} \label{s:combin}
\subsection{}

The stratification of $\kbar n m $ based on bubble-trees leads to several inherent combinatorial structures.

\begin{prop} \label{p:chambers}
The space $\kbar n m$ is tiled by $(m-1)!$ chambers. 
\end{prop}

\begin{proof}
There exist $m!$ orderings of the $m$ boundary particles on $\Di$. By Proposition~\ref{p:psl}, the action of $\psor$ is characterized by fixing one boundary particle at $\infty$. The orderings of the remaining $m-1$ particles on $\partial \Di$ represent unique equivalence classes.  
\end{proof}

There is a combinatorial gluing on the boundaries of these $(m-1)!$ chambers which result in $\kbar n m$, based on the following definition:  A \emph{flip} of a flat bubble in a bubble-tree is obtained by replacing it with its mirror image but preserving the remaining bubbles on the tree.

\begin{thm}
Two codimension $k$ cells, each corresponding to a bubble-tree coming from a diagram of $\Di$ with $k$ nested arcs, are identified in $\kbar n m$ if \emph{flips} along flat bubbles of one diagram result in the other.
\end{thm}

\begin{proof}
The group $\pglr$ acts on flat bubbles and there is a $\Z_2$ component of $\pglr$ which identifies each flat bubble with its mirror image.  Since each chamber of $\kbar n m$ is identified with the cyclic ordering of the $m-1$ boundary particles, a flip identifies faces of one chamber with another.
\end{proof}

\begin{exmp}
Figure~\ref{f:flip}(a) and (d) show the two vertices  from Figure~\ref{f:eye-glue}(b), the corners of the Kontsevich eye.   Parts (b) and (c) show the bubble-trees of these cells, respectively.  In $\kbar 2 1$, both of these trees are identified (glued together) since a \emph{flip} of the flat disk of (b) results in (c).
\end{exmp}

\begin{figure}[h]
\includegraphics[width=\textwidth]{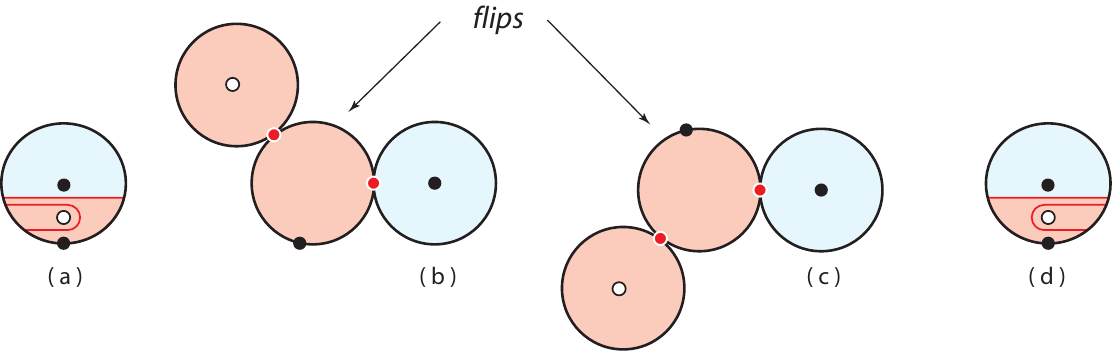}
\caption{Two cells of the eye identified in $\kbar 2 1$ by flips of flat disks.}
\label{f:flip}
\end{figure}

\begin{thm}
The space $\kbar 0 n$ is isomorphic to two disjoint copies of the real moduli space \M{n} of curves, tiled by associahedra $K_{n-1}$.
\end{thm}

\begin{proof}
For $n$ particles on the boundary, the action of $\pslr$ fixes three such particles (call them $0, 1, \infty$) due to M\"obius transformations.  However, there are two equivalence classes of such orderings, with the three fixed particles arranged clockwise as either ${0,1,\infty}$ or as ${1,0,\infty}$.   Each such equivalence class has particles only on the boundary, and is acted upon by flips from $\pglr$.  The result follows from Theorem~\ref{t:realmod}.
\end{proof}

\begin{thm}
The moduli space $\kbar 1 n$ is isomorphic to the space of cyclohedra $\Cyc{n}$.
\end{thm}

\begin{proof}
As $\pslr$ fixes the interior particle and a boundary particle, $\kdel 1 n$ becomes an $(m-1)$-torus, with $m-1$ labeled particles moving on the boundary of $\Di$.  Each ordering of these labels yields a simplicial chamber, all of which glue along the long diagonal of a cube, whose faces identify to form the torus.  The Fulton-MacPherson compactification truncates each simplex into a cyclohedron, resulting in $\Cyc{n}$ promised by Theorem~\ref{t:cycmod}.  
\end{proof}

\begin{rem}
Figure~\ref{f:k4w3}(a) shows the associahedron $K_4$ as a tile of $\kbar 0 5$.  The other tiles of $\kbar 0 5$ correspond to all ways of labeling the five boundary particles.  Similarly, part (b) depicts the cyclohedron $W_3$ as a tile of $\kbar 1 3$.  Compare these diagrams with Figure~\ref{f:k4w3p}.
\end{rem}

\begin{figure}[h]
\includegraphics[width=\textwidth]{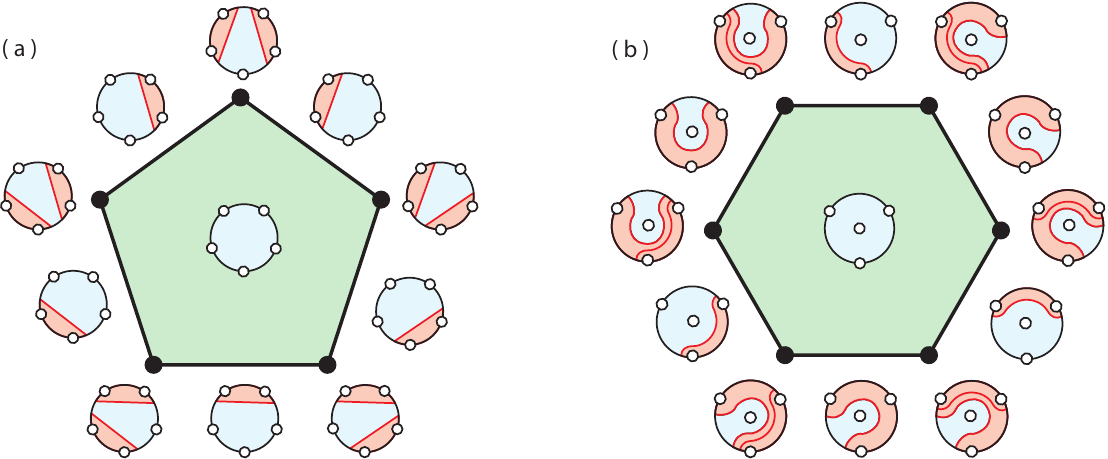}
\caption{(a) Associahedron $K_4$ of $\kbar 0 5$ and (b) cyclohedron $W_3$ of $\kbar 1 3$.}
\label{f:k4w3}
\end{figure}

\begin{exmp}
Figure~\ref{f:k13-tiles}(b) shows two hexagonal cyclohedra $W_3$ tiling $\kbar 1 3$.  Figure~\ref{f:k14-iterate} shows the iterated truncation of the three-torus $\kdel 1 4$, yielding $\kbar 1 4$ tiled by six cyclohedra.
\end{exmp}

\begin{figure}[h]
\includegraphics[width=\textwidth]{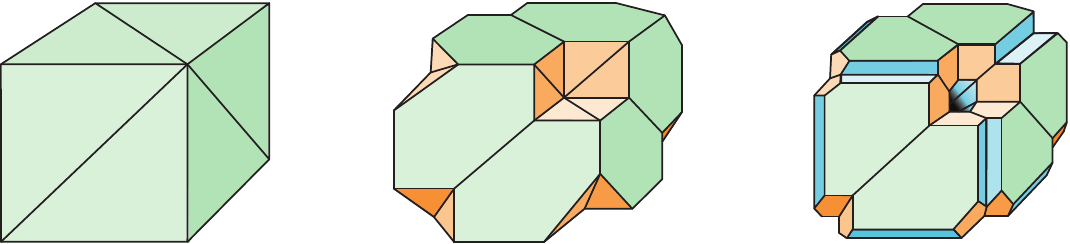}
\caption{Iterated blowups resulting in $\kbar 1 4$ tiled by six cyclohedra $W_4$.}
\label{f:k14-iterate}
\end{figure}

\subsection{}

We close by enumerating the exceptional divisors (the codimension one spaces) of $\kbar m n$.  With this result, one can use induction to calculate all codimension $k$ spaces if desired.

\begin{thm} \label{t:facets}
Let $n,m\geq 1$.  Then the exceptional divisors of \kbar{n}{m} are categorized by the following classification:
\begin{enumerate}

\item  There are a total of $2^n-n-1$ divisors enumerating interior collisions, with $\binom{n}{i}$ divisors where $i>1$ interior particles collide, each topologically equivalent to
$$\kbar{n-i+1}{m} \ \times \ \mnc{i+1}.$$

\item There are a total of $2^m-m-1$ divisors enumerating boundary collisions, with $\binom{m}{b}$ divisors where $b\geq 2$ boundary particles collide, each topologically equivalent to
$$\kbar{n}{m-b+1} \ \times \ \M{b+1}.$$

\item There are a total of $2^m(2^{n-1}-1)$ divisors enumerating mixed collisions, with $\binom{n-1}{i}\binom{m}{b}$ divisors where $1\leq i\leq n-1$ interior particles and $b\geq 0$ boundary particles collide, each topologically equivalent to
$$\kbar{n-i}{m-b+1} \ \times \ \kbar{i}{b+1}.$$
\end{enumerate}
\end{thm}

\begin{proof}
The product structure of these divisors is provided by Theorem~\ref{t:product}.
The exceptional divisors of $\kbar{n}{m}$ are enumerated by particle configurations in which a single collision has occurred, of which there are three types.  For interior collisions, there are $\binom{n}{i}$ ways to choose exactly which $i$ of the $n$ interior particles are in the collision.  For boundary collisions, there are $\binom{m}{b}$ ways to choose exactly which $b$ of the $m$ boundary particles are in the collision.  For mixed collisions, we have chosen our configurations modulo the action of $\pslr$ to be represented by a fixed particle in the interior of the disk, and so this particle may not participate in collisions on the boundary.  There are then $\binom{n-1}{i}$ ways to choose $i$ of the remaining $n-1$ interior particles and $\binom{m}{b}$ ways to choose $b$ of the $m$ boundary particles that are in a mixed collision.
\end{proof}

\begin{rem}
The exceptional divisors corresponding to boundary collisions and interior collisions are in the interior of the moduli space whereas mixed collision divisors are on its boundary.
\end{rem}

\begin{ack}
We thank Eduardo Hoefel, Hiroshige Kajiura, Melissa Liu, and Cid Vipismakul for helpful conversations and clarifications, and a special thanks to Jim Stasheff for his continued encouragement.  We are also grateful to Williams College and to the NSF for partially supporting this work with grant DMS-0353634.
\end{ack}

%
%
\bibliographystyle{amsplain}

\begin{thebibliography}{XX}

\baselineskip=15pt

\bibitem{small03} S.\ Armstrong, M.\ Carr, S.\ Devadoss, E.\ Engler, A.\ Leininger, and M.\ Manapat, ``Particle configurations and Coxeter operads'', \emph{Journal of Homotopy and Related Structures} {\bf 4} (2009) 83-109.

\bibitem{as} S.\ Axelrod and I.\ M.\ Singer, ``Chern-Simons perturbation theory II'', \emph{Journal of Differential Geometry} {\bf 39} (1994) 173-213.

\bibitem{bt} R.\ Bott and C.\ Taubes, ``On the self-linking of knots'', {\em Journal of Mathematical Physics} {\bf 35} (1994) 5247-5287.

\bibitem{cd} M.\ Carr and S.\ Devadoss, ``Coxeter complexes and graph-associahedra'', \emph{Topology and its Applications} {\bf 153} (2006)
2155-2168.

\bibitem{dp} C.\ De Concini and C.\ Procesi, ``Wonderful models of subspace arrangements'', \emph{Selecta Mathematica} {\bf 1} (1995) 459-494.

\bibitem{dev1} S.\ Devadoss, ``Tessellations of moduli spaces and the mosaic operad'', in \emph{Homotopy Invariant Algebraic Structures}, Contemporary  Mathematics {\bf 239} (1999) 91-114.

\bibitem{dev2} S.\ Devadoss, ``A space of cyclohedra'', \emph{Discrete and Computational Geometry} {\bf 29} (2003), 61-75.

\bibitem{dev3} S.\ Devadoss, ``Combinatorial equivalence of real moduli spaces'', \emph{Notices of the American Mathematical Society}
(2004) 620-628.

\bibitem{dhv} S.\ Devadoss, T.\ Heath, and W.\ Vipismakul, ``Deformations of bordered surfaces and convex polytopes'', \emph{Notices of the American Mathematical Society} (2011) 530-541.

\bibitem{dm} S.\ Devadoss and J.\ Morava, ``Diagonalizing the genome I'', preprint arXiv:1009.3224.

\bibitem{fuk} K.\ Fukaya, Y-G Oh, H. Ohta, and K. Ono, ``Lagrangian intersection Floer theory: anomaly and obstruction'', \emph{Kyoto Department of Mathematics} 00-17.

\bibitem{fm} W.\ Fulton and R.\ MacPherson, ``A compactification of configuration spaces'', \emph{Annals of Mathematics} {\bf 139} (1994) 183-225.

\bibitem{gm} A.\ Goncharov and Y.\ Manin\, ``Multiple $\zeta$-motives and moduli spaces \M{n}'', \emph{Compositio Mathematica} {\bf 140} (2004) 1-14.

\bibitem{hoe} E.\ Hoefel, ``OCHA and the Swiss-cheese operad'', \emph{Journal of Homotopy and Related Structures} {\bf 4} (2009) 123-151.

\bibitem{ks} H.\ Kajiura and J.\ Stasheff, ``Open-closed homotopy algebra in mathematical physics'', \emph{Journal of Mathematical Physics} {\bf 47} (2006).

\bibitem{kon} M.\ Kontsevich, ``Deformation quantization of Poisson manifolds'', \emph{Letters in Mathematical Physics} {\bf 66} (2003) 157-216.

\bibitem{km} M.\ Kontsevich and Y.\ Manin, ``Gromov-Witten classes, quantum cohomology, and enumerative geometry'', \emph{Communications in Mathematical Physics} {\bf 164} (1994) 525-562.

\bibitem{lee} C.\ Lee, ``The associahedron and triangulations of the $n$-gon'', \emph{European Journal of Combinatorics} {\bf 10} (1989) 551-560.

\bibitem{git} D.\ Mumford, J.\ Fogarty, and F.\ Kirwan, \emph{Geometric Invariant Theory}, Springer-Verlag, New York, 1994.

\bibitem{pw} T.\ Parker and J.\ Wolfson, ``Pseudo-holomorphic maps and bubble trees'', \emph{Journal of Geometric Analysis} {\bf 3} (1993) 63-98.

\bibitem{rat} J.\ Ratcliffe, \emph{Foundations of hyperbolic manifolds}, Springer-Verlag, New York, 1994.

\bibitem{jds1} J.\ Stasheff, ``Homotopy associativity of $H$-spaces'', \emph{Transactions of the American Mathematical Society} {\bf 108} (1963), 275-292.

\bibitem{tam} D.\ Tamari, ``Monoides pr\'{e}ordonn\'{e}s et chanes de Malcev'', Doctorat \`{e}s-Sciences Math\'{e}matiques Th\`{e}se de Math\'{e}matiques, Universit\'{e} de Paris (1951).

\bibitem{vor} A.\ Voronov, ``The Swiss-cheese operad'', in \emph{Homotopy Invariant Algebraic Structures}, Contemporary Mathematics
{\bf 239} (1999) 365-373.

\end{thebibliography}

\end{document}